\documentclass{aims}
\usepackage{cite}
\usepackage{amsmath}
  \usepackage{paralist}
  \usepackage{graphics} %% add this and next lines if pictures should be in esp format
  \usepackage{epsfig} %For pictures: screened artwork should be set up with an 85 or 100 line screen
\usepackage{graphicx}  \usepackage{epstopdf}%This is to transfer .eps figure to .pdf figure; please compile your paper using PDFLeTex or PDFTeXify.
 \usepackage[colorlinks=true]{hyperref}
\hypersetup{urlcolor=blue, citecolor=red}
\usepackage{comment}
\usepackage{bigints}
\usepackage{todonotes}

\usepackage[normalem]{ulem}

\usepackage{algorithm}
\usepackage{algpseudocode}

\def\currentvolume{X}
 \def\currentissue{X}
  \def\currentyear{200X}
   \def\currentmonth{XX}
    \def\ppages{X--XX}

\newtheorem{theorem}{Theorem}[section]

\newtheorem{lemma}[theorem]{Lemma}
\newtheorem{proposition}{Proposition}

\newtheorem{problem}{Problem}
\theoremstyle{definition}
\newtheorem{definition}[theorem]{Definition}
\newtheorem{remark}{Remark}
\newtheorem{assumption}{Assumption}

\title[Optimal Platoon Coordination] %Use the shortened version of the full title
      {A Constrained Optimal Control Framework for Vehicle Platoons with Delayed Communication}

\author[A M Ishtiaque Mahbub and Behdad Chalaki and Andreas A. Malikopoulos]{}


 \email{mahbub@udel.edu}
 \email{bchalaki@udel.edu}
 \email{andreas@udel.edu}

\thanks{The authors are supported by ARPAE grant DE-AR0000796}

\thanks{$^*$ Corresponding author: A M Ishtiaque Mahbub}

\begin{document}
\emergencystretch 3em % to make sure we are not going over the margin

\maketitle

% Enter the first author's name and address:
\centerline{\scshape A M Ishtiaque Mahbub$^*$, Behdad Chalaki,}
\medskip
{\footnotesize
% please put the address of the first author
 %\centerline{University of Delaware}
   %\centerline{Other lines}
   %\centerline{130 Academy Street, Newark, DE-19716, USA}
} % Do not forget to end the {\footnotesize by the sign }

\medskip

\centerline{\scshape and Andreas A. Malikopoulos}
\medskip
{\footnotesize
 % please put the address of the second  and third author
 \centerline{ University of Delaware}
   %\centerline{Other lines}
   \centerline{130 Academy Street, Newark, DE-19716, USA}
}

\bigskip

% The name of the associate editor will be entered by an editorial staff
% "Communicated by the associate editor name" is not needed for special issue.
% \centerline{(Communicated by the associate editor name)}

%The abstract of your paper
\begin{abstract}
Vehicle platooning using connected and automated vehicles (CAVs) has attracted considerable attention. In this paper, we address the problem of optimal coordination of CAV platoons at a highway on-ramp merging. We present a single-level constrained optimal control framework that optimizes fuel economy and travel time of the platoons while satisfying the state, control, and safety constraints. We also explore the effect of delayed communication among the CAV platoons and propose a robust coordination framework to enforce lateral and rear-end collision avoidance constraints in the presence of bounded delays. We provide a closed-form analytical solution of the optimal control problem with safety guarantees that can be implemented in real time. Finally, we validate the effectiveness of the proposed control framework using a high-fidelity commercial simulation environment.

%This is the abstract of your paper and it should not exceed \textbf{200} words.
\end{abstract}

%%%%%%%%%%%%%%%%%%%%%%%%%%%%%%%%%%%%%%%%%%%%%%%%%%%%%%%%%%%%%%%%%%%%%%%%%%%%%%%%

\section{Introduction} \label{sec:1}

\subsection{Motivation} 
Traffic congestion has increased significantly over the last decade \cite{Schrank2019}.
Bottlenecks such as urban intersections, merging roadways, highway on-ramps, roundabouts, and speed reduction zones along with the driver responses \cite{singh2018critical,Knoop2009} to various disturbances in the transportation network are the primary sources of traffic congestion \cite{Margiotta2011}. 
Emerging mobility systems, e.g., connected and automated vehicles (CAVs), lay the foundation to improve safety and transportation efficiency at these bottlenecks by providing the users the opportunity to better monitor the transportation network conditions and make optimal decisions \cite{Spieser2014, Fagnant2014, wadud2016}. Having enhanced computational capabilities, CAVs can establish real-time communication with other vehicles and infrastructure to increase the capacity of critical traffic corridors, decrease travel time, and improve fuel efficiency and safety \cite{Beaver2020DemonstrationCity,chalaki2021CSM,Mahbub2019ACC, mahbub2020sae-2,zhang2019decentralized}.
However, the cyber-physical nature of emerging mobility systems, e.g., data and shared information through vehicle-to-vehicle (V2V) and vehicle-to-infrastructure (V2I) communication, is associated with significant technical challenges and gives rise to a new level of complexity \cite{Malikopoulos2016b} in modeling and control \cite{Ferrara2018}. 

There have been two major approaches to utilizing the connectivity and automation of vehicles, namely, coordination and platooning. The concept of coordination through different traffic bottlenecks is enabled by the vehicle-to-everything communication protocol among the CAVs and the surrounding infrastructure. On the other hand, real-time computation and automation of CAVs enables safe and comfortable trajectories with extremely short headway in the form of CAV platoon, which consists of a string \cite{Kalle2017} of consecutive CAVs traveling together at a constant headway and speed.

In this paper, we employ the concepts of CAV coordination and platooning to address the problem of minimizing traffic congestion at the traffic bottlenecks in an energy-efficient manner. In particular, we aim at optimally coordinating platoons of CAVs at highway on-ramp merging \cite{Papageorgiou2002} in the presence of bounded delays among platoon leaders while guaranteeing state, control, and safety constraints.
%%%%%%%%%%%%%%%%%%%%%%%%%%%%%%%%%%%%%%%%%%%%%%%%%%%%%%%%%%%%%%
%literature review
\subsection{Literature Review}
%
%literature: vehicle smoothing
CAV coordination is an approach that has been explored to mitigate the speed variation of individual CAVs throughout the transportation network \cite{Ntousakis2016aa}. Early efforts \cite{Levine1966, Athans1969} considered a single string of vehicles that was coordinated through a merging roadway by employing a linear optimal regulator. In $1993$, Varaiya \cite{Varaiya1993} outlined the key features of an automated intelligent vehicle/highway system, and proposed a basic control system architecture. In $2004$, Dresner and Stone \cite{Dresner2004} proposed the use of the reservation scheme to control a signal-free intersection of two roads. Since then, several research efforts \cite{Dresner2008, DeLaFortelle2010, Huang2012, Au2010a,Alonso2011} have extended this approach for coordination of CAVs at urban intersections. 
%\todobehdad{@ishti: most of non-self citations in here are a bit old, should we add more recent citations?}
%\todoIshti{@Behdad: definitely! please feel free to add any. I will also do the same on the final revision.}
%cite our lab's papers
More recently, a decentralized optimal control framework was presented in \cite{Malikopoulos2017, mahbub2020decentralized} for coordinating online CAVs at different traffic scenarios such as on-ramp merging roadways, roundabouts, speed reduction zones and signal-free intersections. The framework uses a hierarchical structure consisting of an upper-level vehicle coordination problem to minimize travel time and a low-level energy minimization problem. The state and control constraints in the coordination problem has been addressed in \cite{Malikopoulos2017, malikopoulos2019ACC,Mahbub2020ACC-1, mahbub2020Automatica-2,chalaki2020TITS,chalaki2020TCST} by incorporating the constraints in the low-level optimization problem, and in \cite{Malikopoulos2019CDC, Malikopoulos2020} by incorporating the constraints in the upper-level optimization problem.
Detailed discussions of the research reported in the literature to date on coordination of CAVs can be found in \cite{Malikopoulos2016a} and \cite{guanetti2018control}.

The aforementioned coordination strategies are vehicle-centric approaches focusing on control of individual CAVs within the network, whereas platooning can leverage the full potential of CAVs to enhance current optimal coordination of CAVs. 
%Literature:Platoon coordination
The concept of platoon formation gained momentum %is a popular system-level approach to address traffic congestion, which gained momentum 
in the 1980s and 1990s as a system-level approach to address traffic congestion \cite{Shladover1991,varaiya1993smart,Rajamani2000} and has been shown to have significant benefits \cite{Kalle2015, Kalle2015a}. Shladover \textit{et al.} \cite{Shladover1991} presented the concept of operating automated vehicles in closely spaced platoons as part of an automated highway system, and pioneered the California Partners for Advanced Transportation Technology (PATH) program to conduct heavy-duty truck platooning from $2001$ to $2003$. Rajamani \textit{et al.} \cite{Rajamani2000} discussed the lateral and longitudinal control of CAVs for the automated platoon formation. From a system point of view, platooning of vehicles yields additional mobility benefits.
It has been shown that capacity at a traffic bottleneck, such as an intersection, can be doubled or even tripled by platooning of vehicles \cite{lioris2017platoons}. Moreover, platooning improves the fuel efficiency of the vehicles due to the reduction of the aerodynamic drag within the platoon, especially at high cruising speeds\cite{alam2010experimental, shida2010short, tsugawa2013,Beaver2021Constraint-DrivenStudy}. Various research efforts in the literature have addressed vehicle platooning at highways to increase fuel efficiency, traffic flow, driver comfort, and safety. To date, there has been a rich body of research focusing on exploring several methods of forming and/or utilizing platoons to improve transportation efficiency \cite{van2017fuel,wang2017developing,johansson2018multi,karbalaieali2019dynamic,yao2019managing,xiong2019analysis,pourmohammad2020platform,ard2020optimizing,bhoopalam2018planning,mahbub2021_platoonMixed}. A detailed discussion on different approaches in vehicle platooning systems at highways can be found in \cite{bergenhem2012platoonoverview, zhang2020truckplatoon, kavathekar2011platoonsurvey}.

%contribution
\subsection{Objectives and Contributions of the Paper}
In this paper, we address the problem of coordinating CAV platoons at a highway on-ramp merging. The main objective is to leverage the key concepts of CAV coordination and platooning, and establish a control framework for platoon coordination aimed at improving network performance while guaranteeing safety.

The key contributions of this paper are
(i) the development of a mathematically rigorous optimal control framework for platoon coordination that completely eliminates stop-and-go driving behavior, and improves fuel economy and traffic throughput of the network,
(ii) the derivation and implementation of the optimal control input in real time that satisfies the state, control, and safety constraints subject to bounded delayed communication, and (iii) the validation of the proposed control framework using a commercial traffic simulator by evaluating its performance compared to a baseline scenario.

%\todobehdad{@Ishti: As I was reading this section, it seemed that we are repeating ourself, so I commented out the following and revised this part slightly. Feel free to revert it if you are not a fan :D}
%\todoIshti{@Behdad: I agree with you. the objective and contribution section  was kind of repeating when I was writing this section. I will probably revise the first line a bit to talk about the general "objective" of the paper without any repeation.}
\begin{comment}
This paper advances the state of the art with the following contributions: 
\begin{enumerate}
    \item We present a single-level optimal control framework that optimizes the fuel economy and travel time of the platoons while satisfying the state, control, and safety constraints. 
    \item We develop a robust coordination framework considering the effect of bounded delayed communication to enforce lateral and rear-end collision avoidance constraints. 
    \item We derive a closed-form analytical solution of the optimal control problem using standard Hamiltonian analysis that can be implemented in real time using leader-follower unidirectional communication topology. 
\end{enumerate}

In short, we seek to establish a rigorous control framework that yields a closed-form analytical solution for the formulation considering system constraints and bounded delayed communication, and thus it is appropriate for real-time implementation on-board the CAVs  \cite{mahbub2020sae-1}
To the best of our knowledge, such approach has not yet been reported in the literature to date.
\end{comment}
%comparison with literature
\subsection{Comparison With Related Work}
To the best of our knowledge, this paper is the first attempt to establish a rigorous constrained optimal control framework for coordination of vehicular platoons at a highway on-ramp merging in the presence of bounded inter-platoon delays. This paper advances the state of the art as follows. First, in contrast to other efforts that neglected state/control constraints \cite{Kumaravel:2021uk,Kumaravel:2021wi, stankovic1997suboptimal}, our framework guarantees satisfaction of all of the state, control, and safety constraints in the system. Second, our framework unlike the several efforts in the literature at highway on-ramp merging scenario \cite{Rios-Torres2,pei2019cooperative,xiao2021decentralized,xuFuguo2021decentralized} does not impose a strict first-in-first-out queuing policy to ensure lateral safety. Third, in this paper, we consider the bounded delay in the inter-platoon communication, which most of the studies in the coordination of vehicular platoons neglect \cite{Kumaravel:2021uk,hoef2019truckPlatoon, chang2020mixedPlatoon}. Finally, our framework yields a closed-form analytical solution while satisfying all of the system constraints, and thus it is appropriate for real-time implementation on-board the CAVs  \cite{mahbub2020sae-1}.
%\todobehdad{@Ishti: feel free to revise the above subsection, also if you could add more references for the last two arguments, that will be perfect.}
%\todoIshti{@behdad: this looks perfect. I will add more references where needed.}
%consider delay

%Optimal coordination of CAV platoons has been addressed in \cite{Kumaravel:2021uk} without considering state, control and safety constraints, where a hierarchical control framework has been adopted. In contrast, we propose a computationally efficient single-level optimization framework capable of satisfying all of the system and safety constraints.

%hierarchical vs. single-level

%structure of the paper
\subsection{Organization of the paper}
The remainder of the paper is organized as follows. In Section \ref{sec:problem_formulation}, we present the modeling framework and formulate the problem. In Section \ref{sec:optimal_control}, we provide a detailed exposition of the optimal control framework and the algorithm to implement the closed-form analytical solution to the constrained optimal control problem. In Section \ref{sec:simul}, we evaluate the effectiveness of the proposed approach in a simulation environment. Finally, we draw conclusions and discuss the next steps in Section \ref{sec:conclusion}. %Additional proofs of selected Lemmas and Theorems can be found in Section VII as an appendix.

%%%%%%%%%%%%%%%%%%%%%%%%%%%%%%%%%%%%
%%%%%%%%%%%%%%%%%%%%%%%%%%%%%%%%%%%%
%SECTION II: Problem Formulation
%%%%%%%%%%%%%%%%%%%%%%%%%%%%%%%%%%%%
%%%%%%%%%%%%%%%%%%%%%%%%%%%%%%%%%%%%
%%%%%%%%%%%%%%%%%%%%%%%%%%%%%%%%%%%%%%%%%%%%%%%%%%%%%%%%%%%%%%%%%%%%
\section{Modeling Framework} \label{sec:problem_formulation}
\begin{figure}[t]
    \centering
\includegraphics[width=0.95\linewidth]{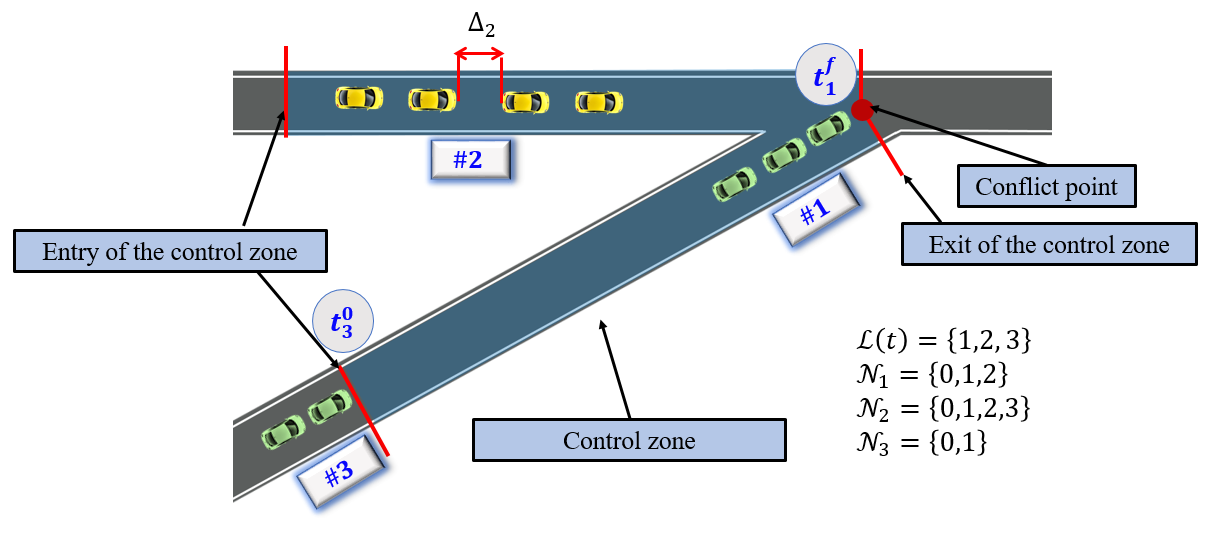}  \caption{On-ramp merging with a single merging point for platoons of CAVs. The control zone is highlighted in light blue color, the entry time and exit time to the control zone are depicted with circles, and example sets of platoon leaders and followers are shown.}
    \label{fig:MG}
\end{figure}
We consider the problem of coordinating platoons of CAVs in a scenario of highway on-ramp merging (Fig. \ref{fig:MG}). Although our analysis can be applied to any traffic scenario, e.g., signal-free intersections, roundabouts, and speed reduction zones, we use a highway on-ramp as a reference to present the fundamental ideas and results of this paper.
%e.g., merging at signal-free urban intersection \cite{Ntousakis:2016aa} and roundabouts, passing through speed reduction zones \cite{Malikopoulos2018c},

The on-ramp merging includes a \textit{control zone}, inside of which platoons of CAVs communicate with the \textit{coordinator}. The coordinator does not make any decision for the CAVs, and only acts as a database for the CAVs. The paths of the main road and the ramp road intersect at a point called \textit{conflict point}, indexed by $n\in\mathbb{N}$, at which lateral collision may occur.
We consider that CAVs have formed platoon upstream of the control zone in a region called \textit{platooning zone}. We refer interested readers to  \cite{mahbub2021_platoonMixed,mahbub2022ACC, tuchner2015vehicle} for further details on platoon formation.

%\todobehdad{I need to revise the figure to match the notation}

%we denote by $p_{i}^n$ and $p_{j}^n$ the distance of the conflict point $n$ from CAV $i$'s and $j$'s paths' entries to the control zone, respectively.  

\subsection{Network Topology and Communication}
%communication topology
\begin{figure}[t]
    \centering
\includegraphics[width=0.95\linewidth]{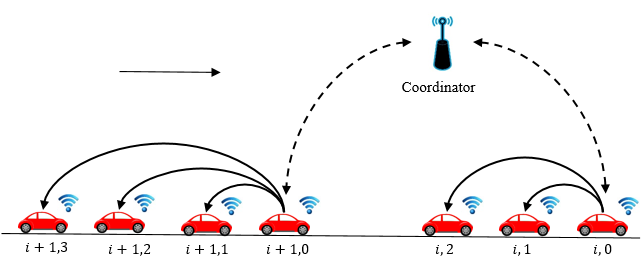}  \caption{Network topology for information flow: (i) bidirectional inter-platoon communication (dashed double-headed arrow) between the platoon leaders via the coordinator, and (ii) unidirectional intra-platoon communication (solid single-headed arrow) from platoon leader to the platoon followers.}
    \label{fig:communication}
\end{figure}

In our modeling framework, we impose the following communication topology based on the standard V2V and V2I communication protocol as shown in Fig. \ref{fig:communication}.
\begin{enumerate}
    \item \textbf{Bidirectional inter-platoon communication:} The leaders of each platoon can exchange information with each other via the coordinator through a V2I communication protocol. The flow of information is bidirectional.
    \item \textbf{Unidirectional intra-platoon communication:} The following CAVs of each platoon can subscribe to the platoon leader's state and control information. The flow of information is unidirectional from the platoon leader to the following CAVs within that platoon.
\end{enumerate}
When a platoon leader enters the control zone, it subscribes to the bidirectional inter-platoon communication protocol to connect with the coordinator and access the information of platoons which are already in the control zone. After obtaining this information, the leader derives its optimal control input (acceleration/deceleration) to cross the control zone without any lateral or rear-end collision with the other CAVs, and without violating any of the state and control constraints. The leader then communicates its derived control input and trajectory information to its followers using the unidirectional intra-platoon communication protocol so that the following CAVs can compute their control input. Finally, the platoon leader transmits its information to the coordinator so that the subsequent platoon leaders can plan their trajectories accordingly.
In this paper, we enhance our framework to consider delayed transmission during the inter-platoon communication protocol due to the physical distance among the platoons. On the other hand, since the CAVs within each platoon are closely spaced, we consider that there is an instantaneous flow of information within the intra-platoon communication protocol. In our modeling framework, we make the following assumption regarding the nature of delay during the inter-platoon communication protocol.

\begin{assumption}\label{assumption: delay}
The communication delay during the bidirectional inter-platoon communication between each platoon leader and the coordinator is bounded and known a priori.
\end{assumption}
Assumption \ref{assumption: delay} enables the determination of upper bounds on the state uncertainties as a result of sensing or communication errors and delays, and incorporate these into more conservative safety constraints, the exposition of which we provide in Section \ref{subsec:delay}.

\subsection{Dynamics and Constraints}
%set definition
Next, we provide some definitions that are necessary in our exposition.
\begin{definition}
The queue that designates the order in which each platoon leader entered the control zone is given by $\mathcal{L}(t)=\{1,\ldots,L(t)\}$, where ${L}(t)\in\mathbb{N}$ is the total number of platoons that are inside the control zone at time $t\in\mathbb{R}_{\geq0}$. When a platoon exits the control zone, its index is removed from $\mathcal{L}(t)$.
\end{definition}

\begin{definition}
CAVs within platoon $i\in\mathcal{L}(t)$ are indexed with set $\mathcal{N}_i=\{0,1,\dots,m_i\}$, where $0$ and $m_i\in\mathbb{N}$ denote the leader and last CAV of the platoon $i$, respectively. The size of each platoon $i\in\mathcal{L}(t)$ is thus the cardinality of set $\mathcal{N}_i$, and denoted by $M_i:=m_i+1$.
\end{definition}

%Let $L(t)\in\mathbb{N}$ be the total number of platoons inside the control zone at time $t\in\mathbb{R}_{\geq0}$, and $\mathcal{L}(t)=\{1,\ldots,L(t)\}$ be the queue that designates the order in which each platoon leader entered the control zone. 
%In what follows, we use a platoon leader index to denote a particular platoon.
%Let $\mathcal{N}_i=\{0,1,\dots,m\}$ index all CAVs inside the platoon $i\in\mathcal{L}(t)$, where $0$ and $m$ denote the leader and last CAV of the platoon $i$. 
%Additionally, Let $\mathcal{N}(t)=\{1,\ldots, N(t)\}$, where $N(t)\in\mathbb{N}$ is the number of CAVs inside the control zone at time $t\in\mathbb{R}^{+}$, be the queue of CAVs inside the control zone. 

In our analysis, we consider that the dynamics of each CAV $j\in\mathcal{N}_i$ in platoon $i\in\mathcal{L}(t)$ is governed by a double integrator,
\begin{align}\label{eq:model2}
\dot{p}_{i,j}(t) &  =v_{i,j}(t), \nonumber\\
\dot{v}_{i,j}(t) &  =u_{i,j}(t), 
\end{align}
where $p_{i,j}(t)\in\mathcal{P}$, $v_{i,j}(t)\in\mathcal{V}$, and
$u_{i,j}(t)\in\mathcal{U}$ denote position, speed, and control input at $t\in\mathbb{R}_{\geq}0$, respectively. The sets $\mathcal{P}, \mathcal{V},$ and $\mathcal{U}$, are compact subsets of $\mathbb{R}$. 
\begin{remark}\label{rem:notation}
In what follows, to simplify notation, we use subscript $i$ instead of ${i,0}$ to denote the leader of platoon $i\in\mathcal{L}(t)$.
\end{remark}
%Let $\mathbf{x}_{i,j}(t)=[p_{i,j}(t)~ v_{i,j}(t)]^\top$ be the state of the $j^{\text{th}}$ CAV at platoon $i$ at time $t$, 
Let $t_{i,0}^0=t_{i}^{0}\in\mathbb{R}_{\geq 0}$ be the time that leader of platoon $i\in\mathcal{N}(t)$ enters the control zone, and $t_{i,0}^{f}=t_{i}^{f}>t_i^0\in\mathbb{R}_{\geq 0}$ be the time that leader of platoon $i$ exits the control zone. Since each CAV $j\in\mathcal{N}_i,~ i\in\mathcal{L}(t)$, has already formed a platoon in the platooning zone, when the leader enters the control zone at time $t_i^0$, we have $v_{i,j-1}(t_i^0)-v_{i,j}(t_i^0) = 0$ and $p_{i,j-1}(t_i^0)-p_{i,j}(t_i^0)-l_c = \Delta_i$, where $l_c$ denote the length of each CAV $j$, and $\Delta_i$ is the safe bumper-to-bumper inter-vehicle gap between CAVs $j,j-1\in\mathcal{N}_i$ within each platoon $i\in\mathcal{L}(t)$. This bumper-to-bumper inter-vehicle gap is imposed by the platoon forming control in platooning zone upstream of the control zone. After exiting the control zone at $t_i^f$, the leader of platoon $i$ cruises with constant speed $v_i(t_i^f)$ until the last follower in the platoon exits the control zone. %Outside the control zone at time , the speed of each platoon leader remains time invariant until all the platoon members exit the control zone and equal to the exit speed $v_i(t_i^f)$. 
Afterwards, each platoon member $j\in\mathcal{N}_i, i\in\mathcal{L}(t)$ is controlled by a suitable car-following model \cite{wiedemann1974} which ensures satisfying rear-end safety constraint. 

For each CAV $j\in\mathcal{N}_i$ in platoon $i\in\mathcal{L}(t)$ the control input and speed are bounded by 
\begin{align}
    u_{\min}&\leq u_{i,j}(t)\leq u_{\max}, \label{eq:uconstraint} \\
    0< v_{\min}&\leq v_{i,j}(t)\leq v_{\max} \label{eq:vconstraint},
\end{align}
where $u_{\min},u_{\max}$ are the minimum and maximum control inputs and $v_{\min},v_{\max}$ are the minimum and maximum speed limit, respectively. 

To ensure rear-end safety between platoon $i\in\mathcal{L}(t)$ and preceding platoon $k\in\mathcal{L}(t)$, we have 
\begin{gather}\label{eq:rearendInterPlatoon}
 p_{k,m_k}(t)-p_i(t)\geq \delta_i(t)= \gamma + \varphi\cdot v_i(t),
\end{gather}
where $m_k$ is the last follower in the platoon $k$ physically located in front of platoon $i$ and $\delta_i(t)$ is the safe speed-dependent distance, while $\gamma$ and $\varphi\in\mathbb{R}_{>0}$ are the standstill distance and reaction time, respectively. 

Similarly, to guarantee rear-end safety within CAVs inside each platoon $i\in\mathcal{L}(t)$, we enforce 
\begin{gather}\label{eq:rearendIntraPlatoon}
 p_{i,j-1}(t)-p_{i,j}(t)\geq \Delta_i+l_c,\quad \forall j\in\{1,\dots ,m_i\}. 
\end{gather}

Finally, let $k\in\mathcal{L}(t)$ correspond to another platoon that has already entered the control zone and may have a lateral collision with platoon $i\in\mathcal{L}(t)$ at conflict point $n$. 
%We denote by $p_{i}^n$ and $p_{j}^n$ the distance of the conflict point $n$ from $i$'s and $k$'s paths' entries, respectively (Fig. \ref{fig:MG}). Since we do not enforce a FIFO queuing policy, platoon $i$ can either cross this conflict point before, or after platoon $k$. 
For the first case in which platoon $i$ reaches the conflict point after platoon $k$, we have
\begin{equation} \label{eq:lateralBefore}
    t_{i}^f - t_{k,m_k}^f \geq t_h,
\end{equation}
where $t_h\in\mathbb{R}_{>0}$ is the minimum time headway between any two CAVs entering node $n$ that guarantees safety, $t_{i}^f$ is the time that leader of platoon $i$ exits the control zone (recall that the conflict point $n$ is at the exit of control zone), and $t_{k,m_k}^f$ is time that the last CAV in the platoon $k$ exits the control zone.
Likewise, for the second case in which platoon $i$ reaches the conflict point $n$ before platoon $k$, we have 
\begin{equation} \label{eq:lateralAfter}
    t_{k}^f - t_{i,m_i}^f \geq t_h.
\end{equation}

\begin{remark}\label{rem:t_i,m^f}
Given the time $t_i^f$ that the platoon leader of platoon $i\in\mathcal{L}(t)$ exits the control zone, we compute the time $t_{i,m_i}^f$ that the last platoon member $m_i\in\mathcal{N}_i$ exits the control zone as
\begin{equation}\label{eq:tfLastVehicle} 
    t_{i,m_i}^f = t_i^f + \frac{(M_i-1)(\Delta_i+l_c)}{v_i(t_i^f)}.
\end{equation}
\end{remark}
To guarantee lateral safety between platoon $i$ and platoon $k$ at a conflict point $n$, either \eqref{eq:lateralBefore} or \eqref{eq:lateralAfter} must be satisfied. Therefore, we impose the following lateral safety constraint on platoon $i$,

\begin{align} 
    \min \Bigg\{ & t_h- (t_{i}^f - t_{k,m_k}^f),~  t_h - (t_{k}^f - t_{i,m_i}^f)  \Bigg\} \leq 0 . \label{eq:lateralMinSafety}
\end{align}

%\todobehdad{Define conflict points}, which is also considered as the exit of control zone for both paths.
%\todobehdad{Control zone's end is defined at the conflict point, after cz's vehicle can follow Acc or cacc or any other car-following model to ensure safety}
%\todobehdad{Define the set of all vehicles, all platoon leaders}
%\todobehdad{speed and control constraint }
%Assumptions
With the state, control and safety constraints defined above, we now impose the following assumption:
\begin{assumption} \label{assumption:platoon}
Upon entering the control zone, the initial state of each CAV $j\in\mathcal{N}_i(t),~i\in \mathcal{L}(t)$, is feasible,
that is, none of the speed or safety constraints are violated.
%None of the state constraints are active for each CAV $j\in\mathcal{N}_i(t),~i\in \mathcal{L}(t)$ at time $t_i^0$ when each platoon $i\in \mathcal{L}(t)$ enters the control zone.
\end{assumption}
This is a reasonable assumption since CAVs are automated; therefore, there is no compelling reason for them to violate any of the constraint by the time they enter the control zone.
% BEHDAD: I revised this, since they can activate the constraint, they only need to not violate it. 
%the constraints by the time they enter the control zone.
%Assumption \ref{assumption:platoon} ensures that, for each CAV $j\in\mathcal{N}_i(t),~i\in \mathcal{L}(t)$, the initial state is feasible. 
%%%%%%%%%%%%%%%%%%%%%%%%%%%%%%%%%%%%%%%%%%%%%%%%%%%%%%%%%%%%%%%%%%%%%%%%%%%%%%%
\subsection{Information Structure}
In this section, we formalize the information structure that is communicated between the CAV leaders and the coordinator inside the control zone.
\begin{definition}\label{def:information_set}
Let $\boldsymbol{\phi}_i$ be the vector containing the parameters of the optimal control policy (formally defined in Section \ref{subsec:leader_control}) of the leader of platoon $i\in\mathcal{L}(t_i^0)$. Then, the \emph{platoon information set} $\mathcal{I}_i$ that the leader of platoon $i$ can obtain from the coordinator after entering the control zone at time $t=t_i^0$ is
\begin{equation}
    \mathcal{I}_i = \{{\boldsymbol{\phi}}_{1:L(t_i^0)-1},~ M_{1:L(t_i^0)},~ t_{1:L(t_i^0)}^0,~ t_{1:L(t_i^0)-1}^f \},
\end{equation}
where  ${\boldsymbol{\phi}}_{1:L(t_i^0)}:=[{\boldsymbol{\phi}}_{1},\ldots, {\boldsymbol{\phi}}_{L(t_i^0)-1}]^T$, $M_{1:L(t_i^0)}:=[{M}_{1},\ldots, {M}_{L(t_i^0)}]^T$,

$t_{1:L(t_i^0)}^f:=[{t}_{1}^0,\ldots, {t}_{L(t_i^0)}^0]^T$ and $t_{1:L(t_i^0)-1}^f:=[{t}_{1}^f,\ldots, {t}_{L(t_i^0)-1}^f]^T$.
\end{definition}
\begin{remark}\label{rem:info_structure_1}
The information structure $\mathcal{I}_i$ for each platoon $i\in\mathcal{L}(t_i^0)$ indicates that the control policy, entry time to the control zone $t_j^0$, exit time of the control zone $t_j^f$, and the platoon size $M_j$ of each platoon $j\in\mathcal{L}(t_i^0)\setminus\{i\}$ already existing within the control zone is available to the leader of platoon $i$ through the coordinator. Note that, although the leader of platoon $i$ knows the endogenous information $t_i^0$ and $M_i$, it needs to compute the vector of its own optimal control input parameters $\boldsymbol{\phi}_i$ and the merging time $t_i^f$, which we discuss in section \ref{sec:optimal_control}.
\end{remark}
%\todobehdad{@Ishti: I think we need to revise the information structure above to avoid any confusion. (1) I think upon entrance to the control zone, $t_i^0$, the leader of platoon $i$ needs to access the whole position trajectory of any other leaders $j$ (by whole, I mean $ \forall t\in [t_j^0, t_j^f] $) which entered earlier than $i$ to the control zone (and thus $j<i$) (2) $M_i$ should not be time variant.(3) for our lateral safety constraint \eqref{eq:lateralAfter} we also need the time that the last CAV in the platoon exits the control zone, i.e., $t_{k,m_k}^f$ for all platoon $k$ in the conflicting path. Honestly I am not sure what is the best way to construct this information set. }

%\todoIshti{@Behdad: (1) I included the initial, final time and control input structure of all the platoons, which should resolve this issue (eliminating dependencies on $[t_i^0, t_i^f]$. (2) You are correct. I corrected the typo. (3) I think the given information set can be used to compute the merging time of the last platoon member. $t_{k,m_k}^f=t_j^f+M_j/v_j(t_j^f)$, where $v_j(t_j^f)$ can now be computed from the given information.}
\begin{definition}
The \emph{member information set} $\mathcal{I}_{i,j}(t)$ that each platoon member $j \in \mathcal{N}_i\setminus\{0\}$ belonging to each platoon $i\in\mathcal{L}(t)$ at time $t\in[t_i^0,t_i^f]$ can obtain is
\begin{align}
    \mathcal{I}_{i,j} = \{ p_{i,0}(t), v_{i,0}(t), u_{i,0}(t)\}.
\end{align}
\end{definition}
\begin{remark}\label{rem:info_structure_2}
The unidirectional intra-platoon communication protocol allows each platoon member $j \in \mathcal{N}_i\setminus\{0\}$ belonging to platoon $i\in\mathcal{L}(t)$ to access the state and control input information of its platoon leader in the form of $\mathcal{I}_{i,j}$ at each time $t\in[t_i^0,t_i^f]$. The set $\mathcal{I}_{i,j}$ is subsequently used to derive the optimal control input $u_{i,j}^*(t)$ of each platoon member $j$, which we discuss in detail in Section \ref{subsec:follower_control}.
\end{remark}

%problem statement
\section{Optimal Coordination Framework}\label{sec:optimal_control}
In what follows, we introduce our coordination framework which consists of two optimal control problems. The first problem is to develop an energy-optimal control strategy for the platoon leaders to minimize their travel time while guaranteeing that none of their state, control, and safety constraints becomes active. The second problem is concerned with the optimal control of followers within each platoon in order to maintain the platoon formation while ensuring safety and string stability. 
%With the above modeling framework and information structure at hand, we next introduce the problems that are addressed in the paper. 

\subsection{Optimal Control of Platoon Leaders}\label{subsec:leader_control}
In this section, we extend the single-level optimization framework we developed earlier for coordination of CAVs in \cite{Malikopoulos2020} to establish a framework for coordinating platoons of CAVs. Upon entrance to the control zone, the leader of platoon $i\in\mathcal{L}(t)$ must determine the exit time $t_i^f$ (recall that based on Remark \ref{rem:notation}, this is the time that the leader of platoon $i$ exits the control zone). The exit time $t_i^f$ corresponds to the unconstrained energy optimal trajectory for the platoon leader ensuring that the resulting trajectory does not activate any of \eqref{eq:model2} - \eqref{eq:rearendInterPlatoon} and \eqref{eq:lateralMinSafety}.
The \emph{unconstrained solution} of the leader of platoon $i$ is given by \cite{Malikopoulos2020}  
\begin{align} \label{eq:optimalTrajectory}
    u_i(t) &= 6 a_i t + 2 b_i, \notag \\
    v_i(t) &= 3 a_i t^2 + 2 b_i t + c_i, \\
    p_i(t) &= a_i t^3 + b_i t^2 + c_i t + d_i, \notag
\end{align}
where $a_i, b_i, c_i, d_i$ are constants of integration.
The leader of platoon $i$ must also satisfy the boundary conditions

\begin{align}
    p_i(t_i^0) &= p_i^0, \quad& v_i(t_i^0) &= v_i^0, \label{eq:IC} \\
    p_i(t_i^f) &= p_i^f, \quad& u_i(t_i^f) &= 0, \label{eq:BC}
\end{align}
where $p_i$ is known at $t_i^0$ and $t_i^f$ by the geometry of the road, and $v_i^0$ is the speed at which the leaders of platoon $i$ enters the control zone.
The final boundary condition, $u_i(t_i^f) = 0$, results from $v_i(t_i^f)$ being left unspecified \cite{bryson1975applied}. There are five unknown variables that determine the optimal trajectory of the leader of the platoon $i$, four constants of integration from \eqref{eq:optimalTrajectory}, and the unknown exit time $t_i^f$. The value of $t_i^f$ guarantees that the unconstrained trajectories in \eqref{eq:optimalTrajectory} satisfy all the state, control, and safety constraints in \eqref{eq:uconstraint}, \eqref{eq:vconstraint} and \eqref{eq:rearendInterPlatoon}, respectively, and the boundary conditions in \eqref{eq:BC}.
In practice, for the leader of each platoon $i\in\mathcal{L}(t)$, the coordinator stores the optimal exit time $t_i^f$ and the corresponding coefficients $a_i, b_i, c_i, d_i$. We denote the coefficients of the optimal control policy for leader of platoon $i\in\mathcal{L}(t)$ by vector $\boldsymbol{\phi}_i =[a_i, b_i, c_i, d_i]^T$, which is an element of platoon information set for the leader of platoon $i\in\mathcal{L}(t)$ (Definition \ref{def:information_set}). 
%It has been shown \cite{Malikopoulos2020} that there is no duality gap in \eqref{eq:tif}, thus the solution can be derived in real time.}
%\todobehdad{@Ishti: I am not sure if we can conclude that since there is no duality gap, we can derive the solution in real time.}
%@behdad: I commented out the duality gap portion.
We formally define our single-level optimization framework for platoon leaders as follows.
\begin{problem}\label{prb:singleLevelProblemPlatoons}
Upon entering the control zone, each leader of platoon $i\in\mathcal{L}(t)$ accesses the information set $\mathcal{I}_i$ and solves the following optimization problem at $t_i^0$
\end{problem}

\begin{align}\label{eq:tif}
    &\min_{t_i^f\in \mathcal{T}_i(t_i^0)} t_i^f \\
    \text{subject to: }& \notag\\
    & \eqref{eq:rearendInterPlatoon}, \eqref{eq:lateralMinSafety},\eqref{eq:optimalTrajectory},  \notag
\end{align}
\textit{where the compact set $\mathcal{T}_i(t_i^0)=[\underline{t}_i^f, \overline{t}_i^f]$ is the set of feasible solution of leader of platoon $i\in\mathcal{N}(t)$ for the exit time that satisfy the boundary conditions without activating the constraints, while $\underline{t}_i^f$ and $\overline{t}_i^f$  denote the minimum and maximum feasible exit time computed at $t_i^0$}. 

\begin{remark}
We can derive the optimal control input of the platoon leaders using the solution of Problem \ref{prb:singleLevelProblemPlatoons}, $t_i^f$, the boundary conditions \eqref{eq:IC}-\eqref{eq:BC} and  \eqref{eq:optimalTrajectory}.
\end{remark}

%introducing formulas for deriving the compact set to avoid confusion for the reviewers
In what follows, we continue our exposition by briefly reviewing the process to compute the compact set $\mathcal{T}_i(t_i^0)$ at time $t_i^0$ using the speed and control input constraints \eqref{eq:uconstraint}-\eqref{eq:vconstraint}, initial condition \eqref{eq:IC}, and final condition \eqref{eq:BC}. Details regarding the derivation of the compact set $\mathcal{T}_i(t_i^0)$ can be found in \cite{chalaki2020experimental}.

The lower-bound $\underline{t}_{i}^f$ of $\mathcal{T}_i(t_i^0)$ can be computed by considering the state and control constraints  and boundary conditions as
 \begin{equation}\label{eq:minbound}
     \underline{t}_{i}^f = \min \left\{ t_{i,u_{\max}}^f  , t_{i,v_{\max}}^f \right \},
 \end{equation}
 where, 
\begin{align*}
         t_{i,v_{\max}}^f &= \frac{3 (p_i(t_i^f)-p_i(t_i^0))}{v_i(t_i^0) + 2 v_{\max}},\\
         t_{i,u_{\max}}^f &= \frac{\sqrt{9 {v_i(t_i^0)}^2 + 12 (p_i(t_i^f)-p_i(t_i^0)) u_{\max}} - 3 {v_i(t_i^0)}}{2 u_{\max}}.
\end{align*}
Here, $t_{i,v_{\max}}^f$ and $ t_{i,u_{\max}}^f$ are the times which leader of platoon $i\in\mathcal{L}(t)$ achieves its maximum speed at the end of control zone and its maximum control input at the entry of the control zone, respectively. 
Similarly, we derive the upper-bound $\overline{t}_{i}^f$ as 
\begin{equation}\label{eq:maxbound}
        \overline{t}_{i}^f = 
        \begin{cases}
            t^f_{i,v_{\min}},&\text{ if }\ 9 {v_i(t_i^0)}^2 + 12 (p_i(t_i^f)-p_i(t_i^0)) u_{\min} < 0,\\
            \max \{ t_{i,u_{\min}}^f  , t_{i,v_{\min}}^f\} ,&\text{otherwise,}\ \\
        \end{cases}
\end{equation}
 where 
\begin{align*}
        t^f_{i,v_{\min}} &= \frac{3 (p_i(t_i^f)-p_i(t_i^0))}{v_i(t_i^0) + 2 v_{\min} },\\
        t^f_{i,u_{\min}} &= \frac{\sqrt{9 {v_i(t_i^0)}^2 + 12 (p_i(t_i^f)-p_i(t_i^0)) u_{\min}} -3 v_i(t_i^0) }{2 u_{\min} }. 
\end{align*}
Similar to the previous case, $t_{i,v_{min}}^f$ and $ t_{i,u_{\min}}^f$ are the times at which leader of platoon $i\in\mathcal{L}(t)$ achieves its minimum speed at the end of control zone and its minimum control input at the entry of the control zone, respectively.

%transition to the next subsection
Note that, the solution to the optimal control problem \ref{prb:singleLevelProblemPlatoons} yields the optimal control input $u_i^*(t)$ for each platoon leader $i\in\mathcal{L}(t)$ for $t\in[t_i^0, t_i^f]$. However, the solution to this problem does not consider the stability criteria of the platoon \cite{naus2010string}, which is essential to guarantee safety within the platooning CAVs. In the following section, we first introduce the notion of stability during platoon coordination, and then propose a control structure $u_{i,j}(t)$ for each platoon member $j\in\mathcal{N}_i\setminus\{0\}$, $i\in\mathcal{L}(t)$ that is optimal subject to constraints, and satisfies the stability properties.

\subsection{Optimal Control of Followers Within Each Platoon}\label{subsec:follower_control}

Stability properties of the platoon system are well discussed in the literature \cite{feng2019string, morbidi2013decentralized, naus2010string}. In general, there are two types of stability: (a) local stability, which describes the ability of each platoon member converging to a given trajectory, and (b) string stability, where any bounded disturbance introduced into the platoon are not amplified while propagating downstream along the vehicle string. In this paper, we adopt the following definition of platoon stability that encompasses the above stability notions \cite{feng2019string}.

\begin{definition}\label{def:string_stability}
A platoon $i\in\mathcal{L}(t)$ is stable if, for any bounded initial disturbances to all the CAVs $j\in\mathcal{N}_i$, the position fluctuations of all the CAVs remain bounded (string stability) and approach zero as time goes to infinity (local stability).
\end{definition}

%Note that, the invariance of the bounds under the platoon length is an important property, which guarantees that the notion of string stability is scalable and allows for the addition or removal of vehicles from a string without affecting the string stability of the platoon. 

With the stability properties Definition \ref{def:string_stability}, we introduce the control problem of each platoon member $j\in\mathcal{N}_i\setminus\{i\}$, $i\in\mathcal{L}(t)$.

\begin{problem}\label{problem:platoon_member_control}
Each platoon member $j\in\mathcal{N}_i\setminus\{0\}$, $i\in\mathcal{L}(t)$ needs to derive its control input $u_{i,j}(t)$ for all $t\in[t_i^0, t_i^f]$ that 
\begin{enumerate}
    \item is energy and time-optimal subject to the state and control constraints in \eqref{eq:uconstraint}-\eqref{eq:vconstraint}, and rear-end collision avoidance constraint in \eqref{eq:rearendIntraPlatoon}, and
    \item satisfying the stability properties according to Definition \ref{def:string_stability}.
\end{enumerate}
\end{problem}

We provide the following proposition that addresses the Problem \ref{problem:platoon_member_control}.
\begin{proposition}\label{prop:platoon_member_control}
For each platoon member $j\in\mathcal{N}_i\setminus\{0\}$ in the platoon $i\in\mathcal{L}(t)$, the optimal control input $u_{i,j}(t)=u^{*}_{i,0}(t)$, where $u^{*}_{i,0}(t)$ is the solution to Problem \ref{prb:singleLevelProblemPlatoons}, is an optimal solution to Problem \ref{problem:platoon_member_control}.
\end{proposition}

Next, we provide the proof of Proposition \ref{prop:platoon_member_control} using the following Lemmas.

\begin{lemma}\label{lem:follower_control_1}
For each platoon member $j\in\mathcal{N}_i\setminus\{0\}$ in each platoon $i\in\mathcal{L}(t)$, the control input $u_{i,j}(t)=u^{*}_{i,0}(t)$  for all $t\in[t_i^0, t_i^f]$ is energy- and time-optimal subject to the control \eqref{eq:uconstraint}, state \eqref{eq:vconstraint} and safety constraint \eqref{eq:rearendIntraPlatoon}.
\end{lemma}

\begin{proof}
(a) Optimality: We derive the control input $u^*_{i,0}(t)$ of the leading CAV $i\in\mathcal{L}(t)$ by solving Problem \ref{prb:singleLevelProblemPlatoons}. The optimal trajectory of the leader of platoon $i\in\mathcal{L}(t)$ is given by \eqref{eq:optimalTrajectory} for all $t\in[t_i^0, t_i^f]$. Thus, for each platoon member $j\in\mathcal{N}_i\setminus\{0\}$, the control input $u_{i,j}(t)$ such that $u_{i,j}(t)=u^{*}_{i,0}(t)$ also generates optimal linear control, quadratic speed and cubic position trajectories as in \eqref{eq:optimalTrajectory}.

(b) Constraint satisfaction: Since $u_{i,j}(t)=u^{*}_{i,0}(t)$, for each platoon member $j\in\mathcal{N}_i\setminus\{0\}$, $i\in\mathcal{L}(t)$, we have $ v_{i,j}(t)=v_{i,0}^*(t)$ for all $t\in[t_i^0, t_i^f]$. The trajectories $v_{i,0}^*(t)$ and $u_{i,0}^*(t)$ do not violate any constraints in \eqref{eq:uconstraint}-\eqref{eq:vconstraint} since they are derived by solving Problem \ref{prb:singleLevelProblemPlatoons}. Therefore, the trajectories $v_{i,j}(t)$ and $u_{i,j}(t)$ of each platoon member $j$ are ensured to satisfy constraints in \eqref{eq:uconstraint}-\eqref{eq:vconstraint}. Additionally, if $ u_{i,j}(t)=u_{i,0}^*(t)$, the inter-vehicle gap $p_{i,j-1}(t)-p_{i,j}(t)-l_c$ between two consecutive platoon members $j,j-1\in\mathcal{N}_i, ~i\in\mathcal{L}(t)$ is time invariant, and equal to $\Delta_i$. %This means that if Assumption \ref{assumption:platoon} holds, then the intra-platoon rear-end collision avoidance constraint in \eqref{eq:rearendIntraPlatoon} is never violated.
Thus, the rear-end safety within CAVs within the platoon $i$ in \eqref{eq:rearendIntraPlatoon} is guaranteed to be satisfied.
\end{proof}

\begin{lemma}\label{lem:follower_control_2}
Each platoon member $j\in\mathcal{N}_i\setminus\{0\}$ for each platoon $i\in\mathcal{L}(t)$ with the control input $u_{i,j}(t)=u^{*}_{i,0}(t)$ for all $t\in[t_i^0, t_i^f]$ is locally stable, and the resulting platoon $i$ is string stable.
\end{lemma}

\begin{proof}
%Suppose that, the state $p_{i,0}(t), v_{i,0}(t)$ and control input $u_{i,0}(t)$ of the leader of platoon $i\in\mathcal{L}(t)$ get updated at time interval $\Delta t$ to be $p_{i,0}(t+\Delta t), v_{i,0}(t+\Delta t), u_{i,0}(t+\Delta t)$. We need to prove that each platoon member $j\in\mathcal{N}_i\setminus\{0\}, ~i\in\mathcal{L}(t)$ is locally stable, and the platoon is string stable subject to the disturbance introduced by the leader.

(a) Local stability: Since $u_{i,j}(t)=u_{i,0}^*(t)$, for each platoon member $j\in\mathcal{N}_i\setminus\{0\}$, $i\in\mathcal{L}(t)$, we have $ v_{i,j}(t)=v_{i,0}^*(t)$ for all $t\in[t_i^0, t_i^f]$. Since there is no communication delay within the unidirectional intra-platoon communication protocol, the speed of each platoon member $j$ converges instantaneously to the speed of the platoon leader $v_{i,0}^*(t)$, which implies local stability. 

(b) String stability: A sufficient condition for the string stability
of a platoon $i\in\mathcal{L}(t)$ containing CAVs $j\in\mathcal{N}_i$ is $\| \frac{u_{i,j}(s)}{u_{i,j-1}(s)} \|_\infty \le 1$ \cite{morbidi2013decentralized}, where $u_{i,j}(s)$ is the Laplace transform of the control input $u_{i,j}(t)$. Since $ u_{i,j}(t)=u_{i,0}^*(t)$ for all $t\in[t_i^0,t_i^f]$, we have $ u_{i,j}(s)=u_{i,0}^*(s)$ for all $j\in\mathcal{N}_i\setminus\{0\}$, which yields $\| \frac{u_{i,j}(s)}{u_{i,j-1}(s)} \|_\infty = 1$. Thus, each platoon $i\in\mathcal{L}(t)$ is string stable.
\end{proof}

\subsection{Delay in Platoon Communication}\label{subsec:delay}
In this section, we enhance our framework to include delay in the bi-directional inter-platoon communication. From Assumption \ref{assumption: delay}, we know that delay is bounded and this bound is known a priori. In particular, suppose the delay in bi-direction communication of platoon leaders takes values in $[\tau_{\min},\tau_{\max}]$, where  $\tau_{\min}\in\mathbb{R}_{\geq 0}$ and $\tau_{\max}\in\mathbb{R}_{\geq 0}$ correspond to the minimum and maximum communication delay, respectively. To account for the effects of communication delays in our framework, we consider the worst-case scenario. Namely, we
consider that it takes $0.5 ~{\tau_{\max}}$ until the coordinator receives the request from the platoon leader, and it takes an extra $0.5~{\tau_{\max}}$ for the leader of platoon $i\in\mathcal{L}(t)$ to receive the platoon information $\mathcal{I}_i$. Thus, the leader needs to cruise with the constant speed that it entered the control zone for $\tau_{\max}$ until it receives the platoon information $\mathcal{I}_i$ to plan its optimal trajectory. After receiving this information, the platoon leader computes the compact set of the feasible solution $\mathcal{T}_i$ at time $t_i^0 + \tau_{\max}$ with initial condition $v_i(t_i^0 + \tau_{\max})$ and $p_i(t_i^0 + \tau_{\max})$. Using the compact set $\mathcal{T}_i(t_i^0 + \tau_{\max})$  of the feasible solution, the leader derives its optimal control policy by solving Problem $\ref{prb:singleLevelProblemPlatoons}$. Then, it sends the computed trajectory at time $t_i^0 + \tau_{\max}$ to the coordinator. In the worst-case scenario, the coordinator receives this information after $0.5{\tau_{\max}}$ at $t_i^0 + 1.5 \tau_{\max}$. To ensure that new arriving platoons have access to this information, we need to have the following constraint on the initial conditions of platoons upon entrance the control zone.
 \begin{proposition}\label{prop:delayCond}
 Let platoons $i$ and $j$, $i, j\in\mathcal{L}(t)$, enter the control zone at time $t_i^0$ and $t_j^0>t_i^0,$ respectively. In the presence of a bi-directional inter-platoon communication delay, which takes value in $[\tau_{\min},\tau_{\max}]$, the optimal trajectory of platoon $i$ is accessible to platoon $j$, if $t_j^0 -t_i^0 \geq \tau_{\max}.$
 \end{proposition}
\begin{proof}
Platoon $i$ computes its optimal trajectory at $t_i^0+\tau_{\max}$, but in the worst-case scenario, due to delay in communication, this information becomes available to the coordinator at $t_i^0 + 1.5~\tau_{\max}$. On the other hand, upon entrance the control zone, platoon $j$ sends a request to the coordinator to receive platoon information $\mathcal{I}_j$. However, the coordinator receives this request at $t_j^0+0.5~\tau_{\max}$. In order to have the optimal trajectory of platoon $i$ accessible to platoon $j$ we need to have $t_j^0+0.5~\tau_{\max}\geq t_i^0 + 1.5~ \tau_{\max}$, and the result follows.
\end{proof}
\begin{remark}
We can ensure that the condition in Proposition $\ref{prop:delayCond}$ holds  by using an appropriate controller in the platooning zone upstream of the control zone.
\end{remark}

\subsection{Implementation of the Optimal Coordination Framework}\label{subsec:algo}
In Sections \ref{subsec:leader_control}, \ref{subsec:follower_control} and \ref{subsec:delay}, we provided the exposition of the intricacies of our proposed control framework for optimal platoon coordination. In this section, we introduce the approach that can be applied to implement this framework in real time.

While entering the control zone at time $t_i^0$, platoon leader $i\in\mathcal{L}(t)$ obtains the platoon information $\mathcal{I}_i$ from the coordinator and solves the optimization problem  \eqref{prb:singleLevelProblemPlatoons} by constructing the feasible set $\mathcal{T}_i(t)$ and iteratively checking the safety constraint. The resulting optimal exit time $t_i^f$ is then used along with the initial \eqref{eq:IC} and boundary \eqref{eq:BC} conditions to derive the vector of control input coefficients $\boldsymbol{\phi}_i$ using \eqref{eq:optimalTrajectory}. Subsequently, each CAV $j\in\mathcal{N}_i$ in platoon $i\in\mathcal{L}(t)$ computes its optimal control input $u_{i,j}(t)$ at each time instance $t\in[t_i^0, t_i^f]$ using $\boldsymbol{\phi}_i$. In what follows, we provide an algorithm that delineates the step-by-step implementation of the proposed optimal platoon coordination framework.

\begin{algorithm}
 \caption{Vehicular Platoons Coordination Algorithm}
% \begin{flushleft}
%        \textbf{Input:} \\
%        \textbf{Output:} 
%\end{flushleft}
%\Statex \textbf{Input:} $l_i^f$ \\
%\hspace*{\algorithmicindent} \textbf{Output:} Arrival time at merging zones $\mathcal{Z}_i = \{z_1,\dots,z_n\}$ 
%\renewcommand{\algorithmicrequire}{\textbf{Input:}}
%\renewcommand{\algorithmicensure}{\textbf{Output:}}

 \begin{algorithmic}[1]
 \For{ $i\in\mathcal{L}(t)$}
 \For{$j\in\mathcal{N}_i$}
 \If{j=0}\Comment{Platoon leader}

\State $u_{i,j} = 0 \quad \forall t\in[t_i^0, t_i^0+\tau_{\max})$ \Comment{Cruise with constant speed}
\State{\texttt{Compute} $\mathcal{T}_i (t_i^0+\tau_{\max})$}
\Comment{Based on \eqref{eq:minbound}-\eqref{eq:maxbound}}
\State{$t_{i}^f, \boldsymbol{\phi}_i \gets$ \textit{Platoon Leader Control()}} \Comment{Algorithm 2}
\State{$[a_i, b_i, c_i, d_i] \gets \boldsymbol{\phi}_i$}
\State{$u_{i,j}(t) \gets 6 a_i t + 2 b_i$} 
\Comment{$\forall t\in[t_i^0+\tau_{\max}, t_i^f]$}
 \Else \Comment{Platoon followers}
 \State { $u_{i,j}(t)=u_{i,0}(t)$}
 \EndIf
 \EndFor
  \EndFor
 %\State \Return $\{t_i^{z^\ast}~|~z\in\mathcal{Z}_i\}$
 \end{algorithmic} \label{Alg:VehicularPlatoonsAlgorithm}
 \end{algorithm}

\begin{algorithm}
 \caption{Platoon Leader Control}
 \begin{flushleft}
        \textbf{Input:} Platoon Information set $\mathcal{I}_i$, Compact feasible set $\mathcal{T}_i (t_i^0+\tau_i)=[\underline{t}_i^f, \overline{t}_i^f]$\\
        \textbf{Output:} Exit time $t_{i}^f$, Coefficients of the optimal control policy $\boldsymbol{\phi}_i $
\end{flushleft}

 \begin{algorithmic}[1]
 \State{$t_{i}^f \gets \underline{t}_i^f$}
\State{$k \gets$ platoon physically located in front of platoon $i$ }
\State{$p_{k,m_k}(t)\gets p_{k}(t) -(M_k -1) (\Delta_k+l_c) $}\Comment{Position of the last follower}
\While{$p_{k,m_k}(t)-p_i(t) < \delta_i(t)$}
\Comment{Rear-end safety}
 \State{$t_{i}^f \gets t_{i}^f+dt$}
\EndWhile
\State{\textit{lateral} $\gets$ list of all platoons $j<i$ from the other road}
\For{$j \in~$ \textit{lateral}}
\State{\texttt{Compute} $t_{j,m_j}^f$ and $t_{i,m_i}^f$ from \eqref{eq:tfLastVehicle}}
\While{$t_{i}^f - t_{j,m_j}^f < t_h$ \texttt{AND} $t_{j}^f - t_{i,m_i}^f < t_h$}\Comment{Lateral safety}
 \State{$t_{i}^f \gets t_{i}^f+dt$}
\EndWhile
\EndFor
\State{\texttt{Compute} $\boldsymbol{\phi}_i$} \Comment{From \eqref{eq:optimalTrajectory}-\eqref{eq:BC}}

 \end{algorithmic} \label{Alg:Platoon LeaderControl}
 \end{algorithm}

\section{Simulation Example}\label{sec:simul}
%\todoIshti{Placeholder for now. To be updated.}
%\todobehdad{1. speed vs distance plot for main road and merging road, in the optimal case and the baseline(same initial condition, but desired speed set as desire speed) 2. Showing headway for different platoons 3. Another baseline could be without the platoon to show the improvement of the platoon forming}
\subsection{Simulation Setup }
To evaluate and validate the performance of our proposed optimal platoon coordination framework, we employ the microscopic traffic simulation software VISSIM v$11.0$ \cite{vissim}. We create a simulation environment with a highway on-ramp merging, which has a control zone of length $560$ m. In our simulation framework, we use VISSIM's \textit{component object model} (COM) interface with Python $2.7$ to generate platoons of CAVs on the main road and on the on-ramp at different time intervals. The time interval between two consecutive platoon generations is randomized with a uniform probability distribution, and the bounds can be controlled to increase or decrease the traffic volume in each roadway. The length of each platoon is also randomly selected from a set of $2$ to $4$ vehicles with equal probability. The speed limit of each roadway is set to be $16.67$ m/s, and the maximum and minimum acceleration limit is $3$ m/s$^2$ and -$3$ m/s$^2$, respectively. Vehicles enter the main road and the on ramp with a traffic volume of $700$ and $650$ vehicle per hour per lane with random initial speed uniformly chosen from a set of $13.89$ to $16.67$  m/s. Videos of the experiment can be found at the supplemental site, \url{https://sites.google.com/view/ud-ids-lab/CAVPLT}.
%
%% USE the EPS figures for greater clarity
%% Use the PNG figures for submission in the arxiv
\begin{figure}[!]
    \centering
\includegraphics[width=0.95\linewidth]{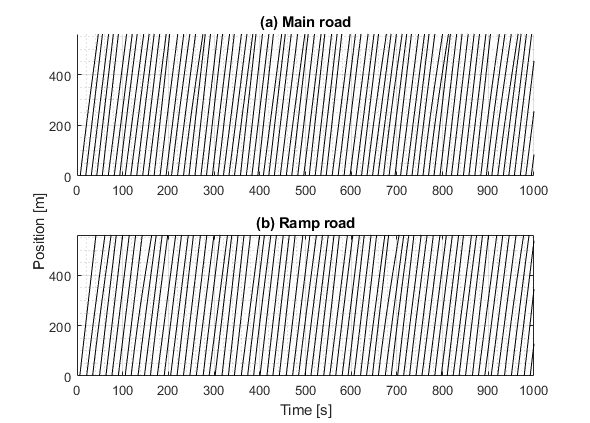}  \caption{The position trajectories of the optimally coordinated CAV platoons at the (a) main road and (b) ramp road are shown.}
    \label{fig:posVsTime}
\end{figure}

\begin{figure}[t]
    \centering
\includegraphics[width=0.95\linewidth]{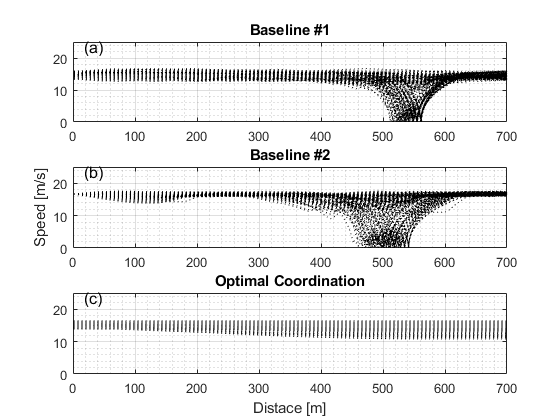}  \caption{Speed profiles of $400$ vehicles traveling through the on-ramp merging scenario for three cases: (a) baseline without platooning, (b) baseline with platooning and (c) optimal platoon coordination.}
    \label{fig:spdVsDist}
\end{figure}

\begin{figure}[t]
    \centering
\includegraphics[width=0.95\linewidth]{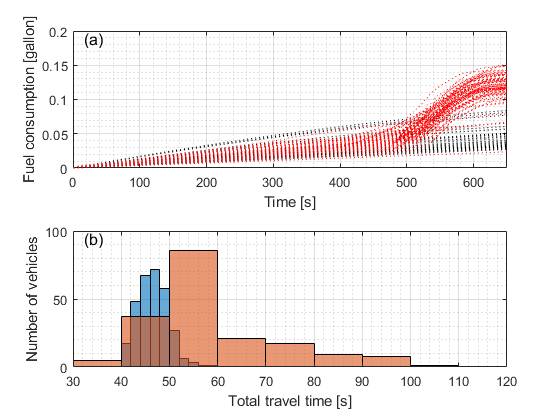}  \caption{Comparison of performance metrics: (a) cumulative fuel consumption of optimal coordination (black) vs. baseline with platooning (red), and (b) total travel time distribution of optimal coordination (blue) vs. baseline with coordination (maroon).}
    \label{fig:FC_time}
\end{figure}

To evaluate the performance of the proposed optimal control framework, we simulate the following control cases.

(a) Baseline 1: All vehicles in the network are human-driven vehicles. In this scenario, the Wiedemann car-following model built in VISSIM \cite{vissim} is applied.
The conflict point of the on-ramp merging scenario has a priority mechanism, where the vehicles on the ramp road are required to yield to the vehicle on the main road within a certain look-ahead distance. Vehicles enter the network individually without forming any platoons.

(b) Baseline 2: Similar to the above case, all the vehicles are human driven vehicles integrated with the Wiedemann car-following model, and follow the priority mechanism set at the conflict point of the on-ramp merging scenario. The difference is that, when vehicles enter the network, they have already formed platoons. Note, we consider this case to simulate the same initial condition of the optimal coordination case, which we discuss next. 

(c) Optimal Coordination: All the vehicles present in the network are connected and automated. They enter the network forming platoons of different sizes and optimize their trajectories based on the optimal coordination framework presented in Section \ref{sec:optimal_control}.

We use COM application programming interface to interact with the VISSIM simulator externally, and implement the proposed optimal coordination framework. At each simulation time step, we use the VISSIM-COM interface to collect the required vehicle attributes from the simulation environment and pass them to the external python script. The external python script implements the proposed single-level optimal control algorithm (Section \ref{subsec:algo}) to compute the optimal control input of each CAV within the control zone. Finally, the speed of each platooning CAV is updated in VISSIM traffic simulator in real-time using the COM interface.

\subsection{Results and Discussion}
In Fig. \ref{fig:spdVsDist}, the position trajectories of the optimal coordinated CAV platoons traveling through the main road and the ramp road of the considered on-ramp merging scenario are shown. The spatial gaps between the trajectory paths indicate that our framework satisfies the rear-end collision avoidance constraint without any violation.

To visualize the performance of the proposed coordination framework in comparison with the baseline cases, we focus on Figs. \ref{fig:spdVsDist} and \ref{fig:FC_time}. In Fig. \ref{fig:spdVsDist}, the speed trajectories of all the vehicles in the network are shown. In Figs. \ref{fig:spdVsDist} (a)-(b), both baseline cases show stop-and-go driving behavior close to the conflict point of the on-ramp merging scenario. In contrast, with the optimal coordination framework, we are able to completely eliminate stop-and-go driving behavior, as shown in Fig. \ref{fig:spdVsDist}(c). The elimination of the stop-and-go driving behavior has associated benefits, namely, the minimization of transient engine operation and travel time, as shown in Fig. \ref{fig:FC_time}. In Fig. \ref{fig:FC_time} (a), the baseline case with vehicle platoons (red) show sudden increase in fuel consumption near the conflict point due to the transient engine operation induced by the stop-and-go driving behavior. In contrast, the cumulative fuel consumption trajectories of the optimally coordinated CAVs (black) remain steady throughout the their path. Note that, we use the polynomial metamodel proposed in \cite{Kamal2013a} to compute the fuel consumption of each vehicle. In Fig. \ref{fig:FC_time} (b), we illustrate the distribution of total travel time of the vehicles for the baseline (maroon) and the optimal coordination (blue) framework. The high variance of the travel time for the baseline case compared to the optimal coordination approach indicates increased traffic throughput of the network.

Finally, we provide the summary of the performance metrics in Table \ref{table:summary}. Based on the simulation, the optimal coordination framework shows significant improvement over the baseline cases in terms of average travel time and fuel consumption.

\begin{table}[h]
\caption{Summary of performance metrics}
    \centering
\begin{tabular}{ |p{5cm}|p{3cm}|p{3cm}| }
\hline
Performance Metrics&Avg. travel time [s] & Avg. fuel consumption [gallon]    \\
\hline
Baseline 1 & 57.33 & 0.042\\
\hline
Baseline 2& 52.79 & 0.05\\
\hline
Optimal Coordination& 46.1 & 0.022\\
\hline
\hline
Improvement (baseline 1) [\%]& 19.6& 46.9 \\
\hline
Improvement (baseline 2) [\%] & 12.7 & 38.2\\
\hline
\end{tabular}
\label{table:summary}
\end{table}

\section{Concluding Remarks}\label{sec:conclusion}
In this paper, we leveraged the key concepts of CAV coordination and platooning, and established a rigorous optimal platoon coordination framework for CAVs that improves fuel efficiency and traffic throughput of the network. We presented a single-level optimal control framework that simultaneously optimizes both fuel economy and travel time of the platoons while satisfying the state, control, and safety constraints. We developed a robust coordination framework considering the effect of delayed inter-platoon communication and derived a closed-form analytical solution of the optimal control problem using standard Hamiltonian analysis that can be implemented in real time using leader-follower unidirectional communication topology. Finally, we validated the proposed control framework using a commercial simulation environment by evaluating its performance. Our proposed optimal coordination framework shows significant benefit in terms of fuel consumption and travel time compared to the baseline cases.

Ongoing work addresses the delay in intra-platoon communication and its implications on platoon stability. A potential direction for future research includes relaxing the assumption of $100$\% CAV penetration considering the inclusion of human-driven vehicles. Investigating the impact of connected vehicle technologies on the size of the components of the vehicle's powertrain \cite{Malikopoulos2013b} and the implications on traveler's decisions \cite{Shaltout2014} is another potential direction for future research.

%For acknowledgements section, please don't number the section, please begin it with \section*{Acknowledgements}
%\section*{Acknowledgments} We would like to thank you for \textbf{following the instructions above} very closely in advance. It will definitely save us lot of time and expedite the process of your paper's publication.

% You may incorporate your references as follows in your main tex file.
% Using BibTex is not recommended but can be handled.

\bibliographystyle{AIMS.bst}
\bibliography{references/IDS_Publications_11122021.bib, references/ref_misc, references/reference_all, references/ref_platoon}

\providecommand{\href}[2]{#2}
\providecommand{\arxiv}[1]{\href{http://arxiv.org/abs/#1}{arXiv:#1}}
\providecommand{\url}[1]{\texttt{#1}}
\providecommand{\urlprefix}{URL }
\begin{thebibliography}{10}

\bibitem{alam2010experimental}
\newblock A.~Al~Alam, A.~Gattami and K.~H. Johansson,
\newblock An experimental study on the fuel reduction potential of heavy duty
  vehicle platooning,
\newblock in \emph{13th international IEEE conference on intelligent
  transportation systems},
\newblock IEEE, 2010,
\newblock 306--311.

\bibitem{Kalle2015a}
\newblock J.~Alam~A.~Martensson and K.~H. Johansson,
\newblock Experimental evaluation of decentralized cooperative cruise control
  for heavy-duty vehicle platooning,
\newblock \emph{Control Engineering Practice}, \textbf{38} (2015), 11--25.

\bibitem{Alonso2011}
\newblock J.~Alonso, V.~Milan\'{e}s, J.~P\'{e}rez, E.~Onieva, C.~Gonz\'{a}lez
  and T.~de~Pedro,
\newblock {Autonomous vehicle control systems for safe crossroads},
\newblock \emph{Transportation Research Part C: Emerging Technologies},
  \textbf{19} (2011), 1095--1110.

\bibitem{ard2020optimizing}
\newblock T.~Ard, F.~Ashtiani, A.~Vahidi and H.~Borhan,
\newblock Optimizing gap tracking subject to dynamic losses via connected and
  anticipative mpc in truck platooning,
\newblock in \emph{2020 American Control Conference (ACC)},
\newblock IEEE, 2020,
\newblock 2300--2305.

\bibitem{Athans1969}
\newblock M.~Athans,
\newblock {A unified approach to the vehicle-merging problem},
\newblock \emph{Transportation Research}, \textbf{3} (1969), 123--133.

\bibitem{Au2010a}
\newblock T.-C. Au and P.~Stone,
\newblock {Motion Planning Algorithms for Autonomous Intersection Management},
\newblock in \emph{AAAI 2010 Workshop on Bridging the Gap Between Task and
  Motion Planning (BTAMP),}, 2010.

\bibitem{Beaver2020DemonstrationCity}
\newblock L.~E. Beaver, B.~Chalaki, A.~M. Mahbub, L.~Zhao, R.~Zayas and A.~A.
  Malikopoulos,
\newblock {Demonstration of a Time-Efficient Mobility System Using a Scaled
  Smart City},
\newblock \emph{Vehicle System Dynamics}, \textbf{58} (2020), 787--804.

\bibitem{Beaver2021Constraint-DrivenStudy}
\newblock L.~E. Beaver and A.~A. Malikopoulos,
\newblock Constraint-driven optimal control of multi-agent systems: A highway
  platooning case study,
\newblock \emph{arXiv:2109.05988}.

\bibitem{bergenhem2012platoonoverview}
\newblock C.~Bergenhem, S.~Shladover, E.~Coelingh, C.~Englund and S.~Tsugawa,
\newblock Overview of platooning systems,
\newblock in \emph{Proceedings of the 19th ITS World Congress, Oct 22-26,
  Vienna, Austria (2012)}, 2012.

\bibitem{Kalle2017}
\newblock B.~Besselink and K.~H. Johansson,
\newblock String stability and a delay-based spacing policy for vehicle
  platoons subject to disturbances,
\newblock \emph{IEEE Transactions on Automatic Control}, \textbf{62} (2017),
  4376--4391.

\bibitem{bhoopalam2018planning}
\newblock A.~K. Bhoopalam, N.~Agatz and R.~Zuidwijk,
\newblock Planning of truck platoons: A literature review and directions for
  future research,
\newblock \emph{Transportation research part B: methodological}, \textbf{107}
  (2018), 212--228.

\bibitem{bryson1975applied}
\newblock A.~E. Bryson and Y.~C. Ho,
\newblock \emph{Applied optimal control: optimization, estimation and control},
\newblock CRC Press, 1975.

\bibitem{chalaki2021CSM}
\newblock B.~Chalaki, L.~E. Beaver, A.~M.~I. Mahbub, H.~Bang and A.~A.
  Malikopoulos,
\newblock A scaled smart city for emerging mobility systems,
\newblock \emph{arXiv:2109.05370}.

\bibitem{chalaki2020experimental}
\newblock B.~Chalaki, L.~E. Beaver and A.~A. Malikopoulos,
\newblock Experimental validation of a real-time optimal controller for
  coordination of cavs in a multi-lane roundabout,
\newblock in \emph{31st IEEE Intelligent Vehicles Symposium (IV)}, 2020,
\newblock 504--509.

\bibitem{chalaki2020TCST}
\newblock B.~Chalaki and A.~A. Malikopoulos,
\newblock Optimal control of connected and automated vehicles at multiple
  adjacent intersections,
\newblock \emph{IEEE Transactions on Control Systems Technology}, 1--13.

\bibitem{chalaki2020TITS}
\newblock B.~Chalaki and A.~A. Malikopoulos,
\newblock Time-optimal coordination for connected and automated vehicles at
  adjacent intersections,
\newblock \emph{IEEE Transactions on Intelligent Transportation Systems},
  1--16.

\bibitem{chang2020mixedPlatoon}
\newblock X.~Chang, H.~Li, J.~Rong, X.~Zhao et~al.,
\newblock Analysis on traffic stability and capacity for mixed traffic flow
  with platoons of intelligent connected vehicles,
\newblock \emph{Physica A: Statistical Mechanics and its Applications},
  \textbf{557} (2020), 124829.

\bibitem{DeLaFortelle2010}
\newblock A.~{de La Fortelle},
\newblock {Analysis of reservation algorithms for cooperative planning at
  intersections},
\newblock \emph{13th International IEEE Conference on Intelligent
  Transportation Systems}, 445--449.

\bibitem{Dresner2004}
\newblock K.~Dresner and P.~Stone,
\newblock {Multiagent traffic management: a reservation-based intersection
  control mechanism},
\newblock in \emph{Proceedings of the Third International Joint Conference on
  Autonomous Agents and Multiagents Systems}, 2004,
\newblock 530--537.

\bibitem{Dresner2008}
\newblock K.~Dresner and P.~Stone,
\newblock A multiagent approach to autonomous intersection management,
\newblock \emph{Journal of artificial intelligence research}, \textbf{31}
  (2008), 591--656.

\bibitem{Fagnant2014}
\newblock D.~J. Fagnant and K.~M. Kockelman,
\newblock The travel and environmental implications of shared autonomous
  vehicles, using agent-based model scenarios,
\newblock \emph{Transportation Research Part C: Emerging Technologies},
  \textbf{40} (2014), 1--13.

\bibitem{vissim}
\newblock M.~Fellendorf and P.~Vortisch,
\newblock Microscopic traffic flow simulator vissim,
\newblock in \emph{Fundamentals of traffic simulation},
\newblock Springer, 2010,
\newblock 63--93.

\bibitem{feng2019string}
\newblock S.~Feng, Y.~Zhang, S.~E. Li, Z.~Cao, H.~X. Liu and L.~Li,
\newblock String stability for vehicular platoon control: Definitions and
  analysis methods,
\newblock \emph{Annual Reviews in Control}, \textbf{47} (2019), 81--97.

\bibitem{Ferrara2018}
\newblock A.~Ferrara, S.~Sacone and S.~Siri,
\newblock \emph{Freeway Traffic Modeling and Control},
\newblock Springer, 2018.

\bibitem{guanetti2018control}
\newblock J.~Guanetti, Y.~Kim and F.~Borrelli,
\newblock Control of connected and automated vehicles: State of the art and
  future challenges,
\newblock \emph{Annual Reviews in Control}, \textbf{45} (2018), 18--40.

\bibitem{hoef2019truckPlatoon}
\newblock S.~V.~D. Hoef, J.~M{\aa}rtensson, D.~V. Dimarogonas and K.~H.
  Johansson,
\newblock A predictive framework for dynamic heavy-duty vehicle platoon
  coordination,
\newblock \emph{ACM Transactions on Cyber-Physical Systems}, \textbf{4} (2019),
  1--25.

\bibitem{Huang2012}
\newblock S.~Huang, A.~Sadek and Y.~Zhao,
\newblock {Assessing the Mobility and Environmental Benefits of
  Reservation-Based Intelligent Intersections Using an Integrated Simulator},
\newblock \emph{IEEE Transactions on Intelligent Transportation Systems},
  \textbf{13} (2012), 1201--1214.

\bibitem{mahbub2022ACC}
\newblock A.~M. {Ishtiaque Mahbub} and A.~A. {Malikopoulos},
\newblock {Platoon Formation in a Mixed Traffic Environment: A Model-Agnostic
  Optimal Control Approach},
\newblock \emph{arXiv e-prints}.

\bibitem{johansson2018multi}
\newblock A.~Johansson, E.~Nekouei, K.~H. Johansson and J.~M{\aa}rtensson,
\newblock Multi-fleet platoon matching: A game-theoretic approach,
\newblock in \emph{2018 21st International Conference on Intelligent
  Transportation Systems (ITSC)},
\newblock IEEE, 2018,
\newblock 2980--2985.

\bibitem{Kamal2013a}
\newblock M.~Kamal, M.~Mukai, J.~Murata and T.~Kawabe,
\newblock {Model Predictive Control of Vehicles on Urban Roads for Improved
  Fuel Economy},
\newblock \emph{IEEE Transactions on Control Systems Technology}, \textbf{21}
  (2013), 831--841.

\bibitem{karbalaieali2019dynamic}
\newblock S.~Karbalaieali, O.~A. Osman and S.~Ishak,
\newblock A dynamic adaptive algorithm for merging into platoons in connected
  automated environments,
\newblock \emph{IEEE Transactions on Intelligent Transportation Systems},
  \textbf{21} (2019), 4111--4122.

\bibitem{kavathekar2011platoonsurvey}
\newblock P.~Kavathekar and Y.~Chen,
\newblock Vehicle platooning: A brief survey and categorization,
\newblock in \emph{International Design Engineering Technical Conferences and
  Computers and Information in Engineering Conference}, vol. 54808, 2011,
\newblock 829--845.

\bibitem{Knoop2009}
\newblock V.~L. Knoop, H.~J. {Van Zuylen} and S.~P. Hoogendoorn,
\newblock {Microscopic Traffic Behaviour near Accidents},
\newblock in \emph{18th International Symposium of Transportation and Traffic
  Theory},
\newblock Springer, New York, 2009.

\bibitem{Kumaravel:2021wi}
\newblock S.~D. Kumaravel, A.~Malikopoulos and R.~Ayyagari,
\newblock Optimal coordination of platoons of connected and automated vehicles
  at signal-free intersections,
\newblock \emph{IEEE Transactions on Intelligent Vehicles}, 1--1.

\bibitem{Kumaravel:2021uk}
\newblock S.~D. Kumaravel, A.~A. Malikopoulos and R.~Ayyagari,
\newblock Decentralized cooperative merging of platoons of connected and
  automated vehicles at highway on-ramps,
\newblock in \emph{2021 American Control Conference (ACC)}, 2021,
\newblock 2055--2060.

\bibitem{Kalle2015}
\newblock J.~Larson, K.-Y. Liang and K.~H. Johansson,
\newblock A distributed framework for coordinated heavy-duty vehicle
  platooning,
\newblock \emph{IEEE Transactions on Intelligent Transportation Systems},
  \textbf{16} (2015), 419--429.

\bibitem{Levine1966}
\newblock W.~Levine and M.~Athans,
\newblock {On the optimal error regulation of a string of moving vehicles},
\newblock \emph{IEEE Transactions on Automatic Control}, \textbf{11} (1966),
  355--361.

\bibitem{lioris2017platoons}
\newblock J.~Lioris, R.~Pedarsani, F.~Y. Tascikaraoglu and P.~Varaiya,
\newblock Platoons of connected vehicles can double throughput in urban roads,
\newblock \emph{Transportation Research Part C: Emerging Technologies},
  \textbf{77} (2017), 292--305.

\bibitem{mahbub2020sae-1}
\newblock A.~M.~I. Mahbub, V.~Karri, D.~Parikh, S.~Jade and A.~Malikopoulos,
\newblock A decentralized time- and energy-optimal control framework for
  connected automated vehicles: From simulation to field test,
\newblock in \emph{SAE Technical Paper 2020-01-0579},
\newblock SAE International, 2020.

\bibitem{mahbub2020sae-2}
\newblock A.~M.~I. Mahbub and A.~Malikopoulos,
\newblock Concurrent optimization of vehicle dynamics and powertrain operation
  using connectivity and automation,
\newblock in \emph{SAE Technical Paper 2020-01-0580},
\newblock SAE International, 2020).

\bibitem{Mahbub2020ACC-1}
\newblock A.~M.~I. Mahbub and A.~A. Malikopoulos,
\newblock Conditions for state and control constraint activation in
  coordination of connected and automated vehicles,
\newblock \emph{Proceedings of 2020 American Control Conference}, 436--441.

\bibitem{mahbub2021_platoonMixed}
\newblock A.~M.~I. Mahbub and A.~A. Malikopoulos,
\newblock {A Platoon Formation Framework in a Mixed Traffic Environment},
\newblock \emph{IEEE Control Systems Letters (LCSS)}, \textbf{6} (2021),
  1370--1375.

\bibitem{mahbub2020Automatica-2}
\newblock A.~M.~I. Mahbub and A.~A. Malikopoulos,
\newblock {Conditions to Provable System-Wide Optimal Coordination of Connected
  and Automated Vehicles},
\newblock \emph{Automatica}, \textbf{131}.

\bibitem{Mahbub2019ACC}
\newblock A.~M.~I. Mahbub, L.~Zhao, D.~Assanis and A.~A. Malikopoulos,
\newblock {Energy-Optimal Coordination of Connected and Automated Vehicles at
  Multiple Intersections},
\newblock in \emph{Proceedings of 2019 American Control Conference}, 2019,
\newblock 2664--2669.

\bibitem{mahbub2020decentralized}
\newblock A.~I. Mahbub, A.~A. Malikopoulos and L.~Zhao,
\newblock Decentralized optimal coordination of connected and automated
  vehicles for multiple traffic scenarios,
\newblock \emph{Automatica}, \textbf{117}.

\bibitem{Malikopoulos2020}
\newblock A.~A. Malikopoulos, L.~E. Beaver and I.~V. Chremos,
\newblock Optimal time trajectory and coordination for connected and automated
  vehicles,
\newblock \emph{Automatica}, \textbf{125}.

\bibitem{Malikopoulos2013b}
\newblock A.~Malikopoulos,
\newblock {Impact of component sizing in plug-in hybrid electric vehicles for
  energy resource and greenhouse emissions reduction},
\newblock \emph{Journal of Energy Resources Technology, Transactions of the
  ASME}, \textbf{135} (2013), 041201.

\bibitem{Malikopoulos2016b}
\newblock A.~A. Malikopoulos,
\newblock A duality framework for stochastic optimal control of complex
  systems,
\newblock \emph{IEEE Transactions on Automatic Control}, \textbf{18} (2016),
  780--789.

\bibitem{Malikopoulos2017}
\newblock A.~A. Malikopoulos, C.~G. Cassandras and Y.~J. Zhang,
\newblock A decentralized energy-optimal control framework for connected
  automated vehicles at signal-free intersections,
\newblock \emph{Automatica}, \textbf{93} (2018), 244 -- 256.

\bibitem{malikopoulos2019ACC}
\newblock A.~A. Malikopoulos and L.~Zhao,
\newblock A closed-form analytical solution for optimal coordination of
  connected and automated vehicles,
\newblock in \emph{2019 American Control Conference (ACC)},
\newblock IEEE, 2019,
\newblock 3599--3604.

\bibitem{Malikopoulos2019CDC}
\newblock A.~A. Malikopoulos and L.~Zhao,
\newblock Optimal path planning for connected and automated vehicles at urban
  intersections,
\newblock in \emph{Proceedings of the 58th IEEE Conference on Decision and
  Control, 2019},
\newblock IEEE, 2019,
\newblock 1261--1266.

\bibitem{Margiotta2011}
\newblock R.~Margiotta and D.~Snyder,
\newblock \emph{{An agency guide on how to establish localized congestion
  mitigation programs}},
\newblock Technical report, U.S. Department of Transportation. Federal Highway
  Administration, 2011.

\bibitem{morbidi2013decentralized}
\newblock F.~Morbidi, P.~Colaneri and T.~Stanger,
\newblock Decentralized optimal control of a car platoon with guaranteed string
  stability,
\newblock in \emph{2013 European Control Conference (ECC)},
\newblock IEEE, 2013,
\newblock 3494--3499.

\bibitem{naus2010string}
\newblock G.~J. Naus, R.~P. Vugts, J.~Ploeg, M.~J. van De~Molengraft and
  M.~Steinbuch,
\newblock String-stable cacc design and experimental validation: A
  frequency-domain approach,
\newblock \emph{IEEE Transactions on vehicular technology}, \textbf{59} (2010),
  4268--4279.

\bibitem{Ntousakis2016aa}
\newblock I.~A. Ntousakis, I.~K. Nikolos and M.~Papageorgiou,
\newblock Optimal vehicle trajectory planning in the context of cooperative
  merging on highways,
\newblock \emph{Transportation Research Part C: Emerging Technologies},
  \textbf{71} (2016), 464--488.

\bibitem{Papageorgiou2002}
\newblock M.~Papageorgiou and A.~Kotsialos,
\newblock {Freeway Ramp Metering: An Overview},
\newblock \emph{IEEE Transactions on Intelligent Transportation Systems},
  \textbf{3} (2002), 271--281.

\bibitem{pei2019cooperative}
\newblock H.~Pei, S.~Feng, Y.~Zhang and D.~Yao,
\newblock A cooperative driving strategy for merging at on-ramps based on
  dynamic programming,
\newblock \emph{IEEE Transactions on Vehicular Technology}, \textbf{68} (2019),
  11646--11656.

\bibitem{pourmohammad2020platform}
\newblock N.~Pourmohammad-Zia, F.~Schulte and R.~R. Negenborn,
\newblock Platform-based platooning to connect two autonomous vehicle areas,
\newblock in \emph{2020 IEEE 23rd International Conference on Intelligent
  Transportation Systems (ITSC)},
\newblock IEEE, 2020,
\newblock 1--6.

\bibitem{Rajamani2000}
\newblock R.~Rajamani, H.-S. Tan, B.~K. Law and W.-B. Zhang,
\newblock {Demonstration of integrated longitudinal and lateral control for the
  operation of automated vehicles in platoons},
\newblock \emph{IEEE Transactions on Control Systems Technology}, \textbf{8}
  (2000), 695--708.

\bibitem{Malikopoulos2016a}
\newblock J.~Rios-Torres and A.~A. Malikopoulos,
\newblock {A Survey on Coordination of Connected and Automated Vehicles at
  Intersections and Merging at Highway On-Ramps},
\newblock \emph{IEEE Transactions on Intelligent Transportation Systems},
  \textbf{18} (2017), 1066--1077.

\bibitem{Rios-Torres2}
\newblock J.~Rios-Torres and A.~A. Malikopoulos,
\newblock {Automated and Cooperative Vehicle Merging at Highway On-Ramps},
\newblock \emph{IEEE Transactions on Intelligent Transportation Systems},
  \textbf{18} (2017), 780--789.

\bibitem{Schrank2019}
\newblock B.~Schrank, B.~Eisele and T.~Lomax,
\newblock \emph{{2019 Urban Mobility Scorecard}},
\newblock Technical report, Texas A\& M Transportation Institute, 2019.

\bibitem{Shaltout2014}
\newblock M.~Shaltout, A.~A. Malikopoulos, S.~Pannala and D.~Chen,
\newblock {Multi-disciplinary decision making and optimization for hybrid
  electric propulsion systems},
\newblock in \emph{2014 IEEE International Electric Vehicle Conference, IEVC
  2014}, 2014.

\bibitem{shida2010short}
\newblock M.~Shida, T.~Doi, Y.~Nemoto and K.~Tadakuma,
\newblock A short-distance vehicle platooning system: 2nd report, evaluation of
  fuel savings by the developed cooperative control,
\newblock in \emph{Proceedings of the 10th International Symposium on Advanced
  Vehicle Control (AVEC)},
\newblock KTH Royal Institute of Technology Loughborough, United Kingdom, 2010,
\newblock 719--723.

\bibitem{Shladover1991}
\newblock S.~E. Shladover, C.~A. Desoer, J.~K. Hedrick, M.~Tomizuka,
  J.~Walrand, W.-B. Zhang, D.~H. McMahon, H.~Peng, S.~Sheikholeslam and
  N.~McKeown,
\newblock {Automated vehicle control developments in the PATH program},
\newblock \emph{IEEE Transactions on Vehicular Technology}, \textbf{40} (1991),
  114--130.

\bibitem{singh2018critical}
\newblock S.~Singh,
\newblock \emph{Critical reasons for crashes investigated in the National Motor
  Vehicle Crash Causation Survey. (Traffic Safety Facts Crash Stats.)},
\newblock Technical Report DOT HS 812 506, March, 2018.

\bibitem{Spieser2014}
\newblock K.~Spieser, K.~Treleaven, R.~Zhang, E.~Frazzoli, D.~Morton and
  M.~Pavone,
\newblock Toward a systematic approach to the design and evaluation of
  automated mobility-on-demand systems: A case study in singapore,
\newblock in \emph{Road vehicle automation},
\newblock Springer, 2014,
\newblock 229--245.

\bibitem{stankovic1997suboptimal}
\newblock S.~S. Stankovi{\'c}, M.~J. Stanojevi{\'c} and D.~D. {\v{S}}iljak,
\newblock Decentralized suboptimal lqg control of platoon of vehicles,
\newblock in \emph{Proc. 8th IFAC/IFIP/IFORS Symp. Transp. Syst.}, vol.~1,
  1997,
\newblock 81--86.

\bibitem{tsugawa2013}
\newblock S.~Tsugawa,
\newblock An overview on an automated truck platoon within the energy its
  project,
\newblock \emph{IFAC Proceedings Volumes}, \textbf{46} (2013), 41--46.

\bibitem{tuchner2015vehicle}
\newblock A.~Tuchner and J.~Haddad,
\newblock Vehicle platoon formation using interpolating control,
\newblock \emph{IFAC-PapersOnLine}, \textbf{48} (2015), 414--419.

\bibitem{van2017fuel}
\newblock S.~Van De~Hoef, K.~H. Johansson and D.~V. Dimarogonas,
\newblock Fuel-efficient en route formation of truck platoons,
\newblock \emph{IEEE Transactions on Intelligent Transportation Systems},
  \textbf{19} (2017), 102--112.

\bibitem{Varaiya1993}
\newblock P.~Varaiya,
\newblock Smart cars on smart roads: problems of control,
\newblock \emph{IEEE Transactions on Automatic Control}, \textbf{38} (1993),
  195--207.

\bibitem{varaiya1993smart}
\newblock P.~Varaiya,
\newblock Smart cars on smart roads: problems of control,
\newblock \emph{IEEE Transactions on automatic control}, \textbf{38} (1993),
  195--207.

\bibitem{wadud2016}
\newblock Z.~Wadud, D.~MacKenzie and P.~Leiby,
\newblock Help or hindrance? the travel, energy and carbon impacts of highly
  automated vehicles,
\newblock \emph{Transportation Research Part A: Policy and Practice},
  \textbf{86} (2016), 1--18.

\bibitem{wang2017developing}
\newblock Z.~Wang, G.~Wu, P.~Hao, K.~Boriboonsomsin and M.~Barth,
\newblock Developing a platoon-wide eco-cooperative adaptive cruise control
  (cacc) system,
\newblock in \emph{2017 ieee intelligent vehicles symposium (iv)},
\newblock IEEE, 2017,
\newblock 1256--1261.

\bibitem{wiedemann1974}
\newblock R.~Wiedemann,
\newblock \emph{Simulation des Strassenverkehrsflusses},
\newblock PhD thesis, Universit{\"a}t Karlsruhe, 1974.

\bibitem{xiao2021decentralized}
\newblock W.~Xiao and C.~G. Cassandras,
\newblock Decentralized optimal merging control for connected and automated
  vehicles with safety constraint guarantees,
\newblock \emph{Automatica}, \textbf{123} (2021), 109333.

\bibitem{xiong2019analysis}
\newblock X.~Xiong, E.~Xiao and L.~Jin,
\newblock Analysis of a stochastic model for coordinated platooning of
  heavy-duty vehicles,
\newblock in \emph{2019 IEEE 58th Conference on Decision and Control (CDC)},
\newblock IEEE, 2019,
\newblock 3170--3175.

\bibitem{xuFuguo2021decentralized}
\newblock F.~Xu and T.~Shen,
\newblock Decentralized optimal merging control with optimization of energy
  consumption for connected hybrid electric vehicles,
\newblock \emph{IEEE Transactions on Intelligent Transportation Systems}.

\bibitem{yao2019managing}
\newblock S.~Yao and B.~Friedrich,
\newblock Managing connected and automated vehicles in mixed traffic by
  human-leading platooning strategy: a simulation study,
\newblock in \emph{2019 IEEE Intelligent Transportation Systems Conference
  (ITSC)},
\newblock IEEE, 2019,
\newblock 3224--3229.

\bibitem{zhang2020truckplatoon}
\newblock L.~Zhang, F.~Chen, X.~Ma and X.~Pan,
\newblock Fuel economy in truck platooning: a literature overview and
  directions for future research,
\newblock \emph{Journal of Advanced Transportation}, \textbf{2020}.

\bibitem{zhang2019decentralized}
\newblock Y.~Zhang and C.~G. Cassandras,
\newblock Decentralized optimal control of connected automated vehicles at
  signal-free intersections including comfort-constrained turns and safety
  guarantees,
\newblock \emph{Automatica}, \textbf{109} (2019), 108563.

\end{thebibliography}

\medskip
% The data information below will be filled by AIMS editorial staff
Received xxxx 20xx; revised xxxx 20xx.
\medskip

%%%%%%%%%%%%%%%%%%%%%%%%%%%%%%%%%%%%%%%%%%%%%%%%%%%%%%%%%%%
\begin{comment}
For a linear platoon systems with the unidirectional communication topology, the system is string stable if the transfer function of outputs between vehicle $j$ and its predecessor $j-1$, denoted as $G_{j,j-1}$, is such that
\begin{align}
    G_{j,j-1}(j\omega)_{H_{\infty}} \le 1, \forall j\in \mathcal{N}_i, i\in\mathcal{L_i}
\end{align}
\end{comment}
\end{document}